\documentclass[12pt, reqno]{amsart}
\usepackage{amsmath, amsthm, amscd, amsfonts, amssymb, graphicx, color}
\usepackage[bookmarksnumbered, colorlinks, plainpages]{hyperref}
\hypersetup{colorlinks=true,linkcolor=red, anchorcolor=green, citecolor=cyan, urlcolor=red, filecolor=magenta, pdftoolbar=true}
\input{mathrsfs.sty}
\newtheorem{theorem}{Theorem}[section]
\newtheorem{lemma}[theorem]{Lemma}
\newtheorem{proposition}[theorem]{Proposition}
\newtheorem{corollary}[theorem]{Corollary}
\theoremstyle{definition}

\theoremstyle{remark}
\newtheorem{remark}[theorem]{Remark}
\numberwithin{equation}{section}
\begin{document}

\title[Solutions of a system of operator equations ]{Solutions of the system of operator equations $BXA=B=AXB$ via $*$-order}

\author[M. Vosough, M.S. Moslehian]{Mehdi Vosough and Mohammad Sal Moslehian}
\address{Department of Pure Mathematics, Ferdowsi University of Mashhad, P. O. Box 1159, Mashhad 91775, Iran.}
\email{vosough.mehdi@yahoo.com}
\address{Department of Pure Mathematics,  Center Of Excellence in Analysis on Algebraic Structures (CEAAS), Ferdowsi University of Mashhad, P. O. Box 1159, Mashhad 91775, Iran.}
\email{moslehian@um.ac.ir and moslehian@member.ams.org}


\begin{abstract}
In this paper, we establish some necessary and sufficient conditions for the existence of solutions to the system of operator equations $ BXA=B=AXB $ in the setting of bounded linear operators on a Hilbert space, where the unknown operator $X$ is called the inverse of $A$ along $B$. After that, under some mild conditions we prove that an operator $X$ is a solution of $ BXA=B=AXB $ if and only if $B \stackrel{*}{ \leq} AXA$, where the $*$-order $C\stackrel{*}{ \leq} D$ means $CC^*=DC^*, C^*C=C^*D$. Moreover we present the general solution of the equation above. Finally, we present some characterizations of $C \stackrel{*}{ \leq} D$ via other operator equations.
\end{abstract}

\maketitle
\section{Introduction and preliminaries} \label{intro-sec}

 Throughout the paper, $ \mathscr{H} $ and $\mathscr{K} $ are complex Hilbert spaces. We denote the space of all bounded linear operators from $ \mathscr{H} $ into $ \mathscr{K} $ by $\mathbb{B}(\mathscr{H}, \mathscr{K})$, and write $\mathbb{B}(\mathscr{H})$ when $ \mathscr{H}= \mathscr{K}.$
 Recall that an operator $ A\in\mathbb{B}(\mathscr{H})$ is positive if $ \langle Ax,x\rangle \geq0 $ for all $ x\in\mathscr{H} $ and then we write $A\geq 0$. We shall write $A>0 $ if $ A $ is positive and invertible. An operator $ A\in\mathbb{B}(\mathscr{H})$ is a generalized projection if $ A^2=A^*$. Let $ \mathscr{S}(\mathscr{H}), \mathscr{Q}(\mathscr{H}), \mathscr{OP}(\mathscr{H}), \mathscr{GP}(\mathscr{H})$ be the set of all self-adjoint operators on $ \mathscr{H}, $ the set of all idempotents, the set of orthogonal projections and the set of all generalized projections on $ \mathscr{H}, $ respectively. For $ A\in \mathbb{B}(\mathscr{H}, \mathscr{K})$, let $\mathscr{R}(A)$ and $\mathscr{N}(A)$ be the range and the null space of $ A, $ respectively. The projection corresponding to a closed subspace $\mathscr{M} $ of $ \mathscr{H} $ is denoted by $ P_{\mathscr{M}}$. The symbol $ A^{-} $ stands for an arbitrary generalized inner inverse of A, that is, an operator $ A^{-} $
satisfying $AA^{-}A=A$. The Moore--Penrose inverse of a closed range operator $A$ is the unique operator $ A^\dagger \in\mathbb{B}(\mathscr{H})$ satisfying the following equations
\begin{eqnarray*}
 AA^\dagger A=A, \qquad A^\dagger AA^\dagger=A^\dagger, \qquad (AA^\dagger)^*=AA^\dagger, \qquad (A^\dagger A)^*=A^\dagger A.
\end{eqnarray*}
Then, $A^*AA^\dagger=A^*=A^\dagger AA^*$ and we have the following properties
\begin{eqnarray}\label{m1}
\mathscr{R}(A^\dagger)=\mathscr{R}(A^*)=\mathscr{R}(A^\dagger A)=\mathscr{R}(A^*A), \qquad \mathscr{N}(A^\dagger)=\mathscr{N}(A^*)=\mathscr{N}(AA^\dagger), \nonumber \\
\mathscr{R}(A)=\mathscr{R}(A A^\dagger)=\mathscr{R}(AA^*), \qquad P_{\mathscr{R}(A)}=AA^\dagger {\rm ~and~} P_{\overline{\mathscr{R}(A^*)}}=A^\dagger A.
\end{eqnarray}
For $ A, B\in\mathscr{S}(\mathscr{H})$, $A\leq B $ means $ B-A \geq 0 $. The order $\leq $ is said to be the L\"{o}wner order on $ \mathscr{S}(\mathscr{H})$. If there exists $ C\in \mathscr{S}(\mathscr{H})$ such that $ AC=0 $ and $ A+C=B $, then we write $A\preceq B$. The order $\preceq $ is said to be the logic order on $ \mathscr{S}(\mathscr{H})$.
 For $ A, B\in\mathbb{B}(\mathscr{H}),$ let $ A\stackrel{*}{ \leq}B $ mean\\
 \begin{eqnarray}\label{m2}
 AA^*=BA^*, \quad A^*A=A^*B.
 \end{eqnarray}
It is known that, for $ A,B\in\mathscr{S}(\mathscr{H}), A\preceq B $ if and only if $ A\stackrel{*}{\leq}B $; see \cite{Deng}. We denote by $ A\stackrel{*}{\wedge}B $ the infimum (or the greatest lower bound) of $A$ and $B$ over the $*- $ order and $A\stackrel{*}{\vee}B $ the supremum (or the least upper bound) of $A$ and $B$ over the $*- $ order, if they exist; cf. \cite{Zhang}.\\
 It is known that if $ A\in \mathbb{B}(\mathscr{H},\mathscr{K})$ has closed range, then by considering
 \begin{eqnarray*}
 \mathscr{H}=\mathscr{R}(A^*)\oplus \mathscr{N}(A) {\rm ~and~} \mathscr{K}=\mathscr{R}(A)\oplus \mathscr{N}(A^*)
 \end{eqnarray*}
 we can write\\
 \begin{eqnarray}\label{m3}
 A=\left [  \begin{array}{cc}
 A_{1} & 0\\
 0 & 0
 \end{array}  \right]:
 \left [  \begin{array}{cc}
 \mathscr{R}(A^*)\\
 \mathscr{N}(A)
 \end{array}  \right]
 \rightarrow
 \left [  \begin{array}{cc}
 \mathscr{R}(A)\\
 \mathscr{N}(A^*)
 \end{array}  \right],
 \end{eqnarray}
 where $A_{1}: \mathscr{R}(A^*)\rightarrow \mathscr{R}(A)$ is invertible; see \cite[Lemma 2.1]{Djordjevic}.\\
 Therefore, the Moore--Penrose generalized inverse of $A$ can be represented as
\begin{eqnarray}
 A^\dagger = \left [  \begin{array}{cc}\label{m4}
 A_{1}^{-1} & 0 \\
 0 & 0
 \end{array}  \right]:
 \left [  \begin{array}{cc}
 \mathscr{R}(A) \\
 \mathscr{N}(A^*)
 \end{array}  \right] \rightarrow
 \left [  \begin{array}{cc}
 \mathscr{R}(A^*) \\
 \mathscr{N}(A)
 \end{array}  \right].
 \end{eqnarray}
Many results have been obtained on the solvability of equations for matrices and operators on Hilbert spaces and Hilbert $ C^*- $ modules. In 1976, Mitra \cite{Mitra} considered the matrix equations $ AX=B, AXB=C $ and the system of linear equations $ AX=C, XB=D$. He got the necessary and sufficient conditions for existence and expressions of general Hermitian solutions. In 1966, the celebrated Douglas Lemma was established in \cite{Douglas}. It gives some conditions for the existence of a solution to the equation $ AX=B $ for operators on a Hilbert space. Using the generalized inverses of operators, in 2007, Daji\'{c} and Koliha \cite{Koliha} got the existence of the common Hermitian and positive solutions to the system $ AX=C, XB=D $ for operators acting on a Hilbert space. In 2008, Xu \cite{Xu} extended these results to the adjointable operators. Several general operator equations and systems in some general settings such as Hilbert $C^*$-modules have been studied by some mathematicians; see, e.g., \cite{Farid, Wang1, Moslehian, Wang2}.\\
The matrix equation $ AXB=C $ is consistent if and only if $ AA^{-}CB^{-}B=C $ for some $A^{-}, B^{-}$, and the general solution is $ X=A^{-}CB^{-}+Y-A^{-}AYBB^{-} $, where $Y$ is an arbitrary matrix; see \cite{Mitra}. In 2010, Gonzalez \cite{Gonzalez} got some necessary and sufficient conditions for existence of a solution to the equation $ AXB=C $ for operators on a Hilbert space.\\
Let $ A, B $ or $C$ have closed range. Then, the operator equation $ AXB=C $ is solvable if and only if $ \mathscr{R}(C)\subseteq \mathscr{R}(A)$ and $\mathscr{R}(C^*)\subseteq \mathscr{R}(B^*)$; see \cite[Theorem 3.1]{Gonzalez}. Therefore, if $A$ or $C$ has closed range, then the equation $ AXC=C $ is solvable if and only if $ \mathscr{R}(C)\subseteq \mathscr{R}(A) $, and $CXA=C $ is solvable if and only if $ \mathscr{R}(C^*)\subseteq \mathscr{R}(A^*)$. Deng \cite{Deng1} investigated the equation $ CAX=C=XAC $, which is essentially different from ours. In this paper, we first characterize the existence of solutions of the system of operator equations  $ BXA=B=AXB $ by means of $ *- $ order. After that, we generalize the solutions to the system of operator equations $ BXA=B=AXB $ in a new fashion.

\section{The existence of solutions of the system $BXA=B=AXB$}

We start our work with the celebrated Douglas lemma.
 \begin{lemma}[Douglas Lemma]\cite{Douglas}\label{Douglas}
 Let $ A,C\in \mathbb{B}(\mathscr{H})$. Then, the following statements are equivalent:\\
 \begin{enumerate}
\item[(a)] $\mathscr{R}(C)\subseteq\mathscr{R}(A)$.\\
\item[(b)] There exists $ X\in \mathbb{B}(\mathscr{H})$ such that $AX=C$.\\
 \item[(c)] There exists a positive number $ \lambda $ such that $ CC^*\leq \lambda^2AA^*$.\\
 \end{enumerate}
If one of these conditions holds, then there exists a unique solution $ \widetilde{X}\in\mathbb{B}(\mathscr{H})$ of the equation $ AX=C $ such that $ \mathscr{R}(\widetilde{X})\subseteq \overline{\mathscr{R}(A^*)} $ and $\mathscr{N}(\widetilde{X})=\mathscr{N}(C)$.
\end{lemma}
\begin{lemma}\label{cc}
Let $ A, B\in\mathbb{B}(\mathscr{H})$. If $ \mathscr{R}(B)\subseteq \mathscr{R}(A)$ and $\mathscr{R}(B^*)\subseteq \mathscr{R}(A^*)$, then $ B=B_{1}\bigoplus 0, $ where $ B_{1}\in\mathbb{B}(\overline{\mathscr{R}(A^*)}, \overline{\mathscr{R}(A)})$.
\end{lemma}
\begin{proof}
Let $ A, B $ be operators from the decomposition $ \mathscr{H}=\overline{\mathscr{R}(A^*)}\bigoplus \mathscr{N}(A)$ into the decomposition $ \mathscr{H}=\overline{\mathscr{R}(A)}\bigoplus \mathscr{N}(A^*)$. If $ \mathscr{R}(B)\subseteq \mathscr{R}(A)$, then, by Lemma \ref{Douglas}, there exists $ C\in\mathbb{B}(\mathscr{H})$ such that $ B=AC $ and $ \mathscr{N}(C)=\mathscr{N}(B)$. Since $\mathscr{R}(B^*)\subseteq \mathscr{R}(A^*)$, so $ \mathscr{R}(C^*)\subseteq\overline{\mathscr{R}(C^*)}=\overline{\mathscr{R}(B^*)}\subseteq \overline{\mathscr{R}(A^*)}=\mathscr{N}(P_{\mathscr{N}(A)})$. Hence, $P_{\mathscr{N}(A)}C^*=0$ and so $ CP_{\mathscr{N}(A)}=0 $. It follows from $ \mathscr{N}(C)=\mathscr{N}(B)$ that $ BP_{\mathscr{N}(A)}=0$. \\
If $ \mathscr{R}(B^*)\subseteq \mathscr{R}(A^*)$, then a similar reasoning shows that $ P_{\mathscr{N}(A^*)}B=0$. Therefore,
$P_{\overline{\mathscr{R}(A)}}BP_{\mathscr{N}(A)}=P_{\mathscr{N}(A^*)}BP_{\overline{\mathscr{R}(A^*)}}=P_{\mathscr{N}(A^*)}BP_{\mathscr{N}(A)}=0 $. Hence, $B=B_{1}\bigoplus 0$, where $ B_{1}=P_{\overline{\mathscr{R}(A)}}BP_{\overline{\mathscr{R}(A^*)}} $.
\end{proof}
 \begin{theorem}\label{CXA=C=AXC}
Let $ A\in \mathbb{B}(\mathscr{H})$ and $B\in\mathscr{S}(\mathscr{H})$. If $A$ has closed range, then the following statements are equivalent:
\begin{enumerate}
\item[(1)] The system of operator equations $ BXA=B=AXB $ is solvable;\\
\item[(2)] $AA^\dagger BA^\dagger A=B$;\\
\item[(3)] $ \mathscr{R}(B)\subseteq \mathscr{R}(A)$ and $\mathscr{R}(B)\subseteq \mathscr{R}(A^*)$.\\
 \end{enumerate}
 \end{theorem}
 \begin{proof}
 $ (1) \longrightarrow (2):$ Using (\ref{m1}) and $B=BXA,$ we get that $ \mathscr{R}(B)\subseteq\mathscr{R}(A^*)= \mathscr{R}(A^\dagger A)$. Hence, by Lemma \ref{Douglas}, there exists $ C^*\in \mathbb{B}(\mathscr{H})$ such that $ B=A^\dagger AC^*$. Hence, $ B=CA^\dagger A $. Applying (\ref{m1}) and $AXB =B, $ we derive that $ \mathscr{R}(B)\subseteq\mathscr{R}(A)=\mathscr{R}(AA^\dagger)$. Thus, by Lemma \ref{Douglas}, there exists $ \widetilde{C}\in \mathbb{B}(\mathscr{H})$ such that $ B=AA^\dagger \widetilde{C}. $ It follows that
 \begin{eqnarray*}
 AA^\dagger BA^\dagger A=AA^\dagger (AA^\dagger \widetilde{C})A^\dagger A=AA^\dagger \widetilde{C}A^\dagger A=BA^\dagger A=(CA^\dagger A)A^\dagger A=CA^\dagger A=B.
 \end{eqnarray*}
(2) $\longrightarrow$ (3): Let $ AA^\dagger BA^\dagger A=B$. Then, $ \mathscr{R}(B)\subseteq \mathscr{R}(A)$. It follows from $ B=B^*=(AA^\dagger BA^\dagger A)^*=A^\dagger A BA A^\dagger $ and (\ref{m1}) that $ \mathscr{R}(B)\subseteq \mathscr{R}(A^\dagger)=\mathscr{R}(A^*)$.\\
(3) $\longrightarrow$ (1): Let $ \mathscr{R}(B)\subseteq \mathscr{R}(A)$ and $\mathscr{R}(B)\subseteq\mathscr{R}(A^*)$. Upon applying Lemma \ref{cc}, $ B=B_{1}\bigoplus 0$, where $ B_{1}=P_{\overline{\mathscr{R}(A)}}BP_{\overline{\mathscr{R}(A^*)}} $. Since $A$ has closed rang, so by using (1.3) and (1.4) we have
\begin{eqnarray*}
 A=\left [  \begin{array}{cc}
 A_{1} & 0\\
 0 & 0
 \end{array}  \right] {\rm~ and~ }
 A^\dagger = \left [  \begin{array}{cc}
 A_{1}^{-1} & 0 \\
 0 & 0
 \end{array}  \right].
\end{eqnarray*}
Hence, $AA^\dagger B=B $ and $BA^\dagger A=B$. Thus $ X=A^\dagger $ is a solution of the system $ BXA=B=AXB$.
\end{proof}
\begin{proposition}\label{main}
Let $ A, B, X\in\mathbb{B}(\mathscr{H})$. Then,\\
\begin{eqnarray*}
\mathscr{R}(A)\subseteq \mathscr{R}(B), \quad \mathscr{N}(B)\subseteq \mathscr{N}(A) {\rm ~and~} BXA=B=AXB
 \end{eqnarray*} if and only if
 \begin{eqnarray*}
\mathscr{N}(B)=\mathscr{N}(A), \quad \mathscr{R}(B)=\mathscr{R}(A) {\rm ~and~} AXA=A.
\end{eqnarray*}
\end{proposition}
\begin{proof}
$ (\Longrightarrow): $ Suppose that $\mathscr{R}(A)\subseteq \mathscr{R}(B), \mathscr{N}(B)\subseteq \mathscr{N}(A)$ and $BXA=B=AXB $. It follows from $ BXA=B $ and $\mathscr{N}(B)\subseteq \mathscr{N}(A)$ that $ \mathscr{N}(A)\subseteq\mathscr{N}(B) \subseteq \mathscr{N}(A)$. Hence,
 $ \mathscr{N}(A)=\mathscr{N}(B)$. It follows from $ AXB=B $ and $\mathscr{R}(A)\subseteq\mathscr{R}(B)$ that $ \mathscr{R}(A)\subseteq\mathscr{R}(B)\subseteq\mathscr{R}(A)$. Therefore, $ \mathscr{R}(A)=\mathscr{R}(B)$. Moreover, $ (I-AX)B=0 $ and $\mathscr{R}(A)\subseteq\mathscr{R}(B)$ Hence, we derive that $ (I-AX)A=0 $. So, $ AXA=A $.\\
$ (\Longleftarrow)$: Suppose that $\mathscr{N}(B)=\mathscr{N}(A), \mathscr{R}(B)=\mathscr{R}(A)$ and $AXA=A$. Hence,
\begin{eqnarray*}
(I-AX)A=0 \Longrightarrow \mathscr{R}(A) \subseteq \mathscr{N}(I-AX) \Longrightarrow \mathscr{R}(B) \subseteq \mathscr{N}(I-AX) \Longrightarrow B=AXB,\\
A(I-XA)=0 \Longrightarrow \mathscr{R}(I-XA) \subseteq \mathscr{N}(A) \Longrightarrow \mathscr{R}(I-XA) \subseteq \mathscr{N}(B) \Longrightarrow B=BXA.
\end{eqnarray*}
\end{proof}

\section{System of operator equations $ BXA=B=AXB $ via $*$-order}
 We know that $ (\mathbb{B}(\mathscr{H}), \stackrel{*}{ \leq})$ is a partially ordered set; see \cite{Antezana}. Let $ G_{1}, G_{2}\in\mathbb{B}(\mathscr{H})$ be invertible and $ G_{1}\stackrel{*}{ \leq}A, G_{2}\stackrel{*}{ \leq}A$. Then, $ G_{1}G_{1}^*=AG_{1}^* $ and $G_{2}G_{2}^*=AG_{2}^* $. Hence, we obtain $ G_{1}=G_{2}=A$. This fact leads us to consider the characterizations of $ A\stackrel{*}{ \leq}B $. Now we state the necessary and sufficient conditions in which the common $ *- $ lower or $ *- $ upper bounds of $A$ and $B$ exist.\\
We need the following essential lemmas.
\begin{lemma}\label{order} \cite[Lemma 2.1]{Fang} Let $ A, B\in\mathbb{B}(\mathscr{H})$ and $\overline{\mathscr{M}} $ denote the closure of a space $\mathscr{M} $.
\begin{enumerate}

\item[(a)] $ AA^*=BA^*\Longleftrightarrow A=BP_{\overline{\mathscr{R}(A^*)}}\Longleftrightarrow A=BQ $ for some $Q\in\mathscr{OP}(\mathscr{H})$;\\
\item[(b)] $ A^*A=A^*B\Longleftrightarrow A=P_{\overline{\mathscr{R}(A)}}B\Longleftrightarrow A=PB $ for some $P\in\mathscr{OP}(\mathscr{H})$;\\
\item[(c)] $ A\stackrel{*}{\leq} B \Longleftrightarrow B=A+P_{\mathscr{N}(A^*)}BP_{\mathscr{N}(A)}$;\\
\item[(d)] $ A\stackrel{*}{ \leq}B\Longleftrightarrow A=P_{\overline{\mathscr{R}(A)}}B=BP_{\overline{\mathscr{R}(A^*)}}=P_{\overline{\mathscr{R}(A)}}BP_{\overline{\mathscr{R}(A^*)}}$;\\
\item[(e)] $ A\stackrel{*}{ \leq}B \Longleftrightarrow A=A_{1}\bigoplus0, B=A_{1}\bigoplus B_{1} $;
\end{enumerate}
where $A_{1}\in\mathbb{B}(\overline{\mathscr{R}(A^*)}, \overline{\mathscr{R}(A)}), B_{1}\in\mathbb{B}(\mathscr{N}(A), \mathscr{N}(A^*))$ and $A\bigoplus B $ means the block matrix $\left [  \begin{array}{cc}
A&0\\
0&B
\end{array}  \right]$.
\end{lemma}
The following Lemma is a version of Lemma \ref{Douglas} when the operator $ A $ has closed range.
\begin{lemma}\label{AX=C}\cite[Theorem 3.1]{Koliha}. Let $ A\in\mathbb{B}(\mathscr{H})$ have closed range. Then, the equation $ AX=C $ has a solution $ 	X\in\mathbb{B}(\mathscr{H})$ if and only if $ AA^\dagger C=C $, and this if and only if $ \mathscr{R}(C)\subseteq \mathscr{R}(A)$. In this case, the general solution is
$ X=A^\dagger C+(I-A^\dagger A)T,$ where $T\in\mathbb{B}(\mathscr{H})$ is arbitrary.\\
\end{lemma}
\begin{proposition}\label{star}
Let $ A, B\in\mathbb{B}(\mathscr{H})$. Then, \\
\begin{enumerate}
\item[(a)] If $A$ has closed range and $B\stackrel{*}{\leq}A $, then $ X=A^\dagger $ is a solution of the system $ BXA=B=AXB $.\\
\item[(b)] If $B$ has closed range and $B\stackrel{*}{\leq}A $, then $ X=B^\dagger $ is a solution of the system $ BXA=B=AXB $.
\end{enumerate}
\end{proposition}
\begin{proof}
(a) Let $A$ be a closed range operator and $B\stackrel{*}{\leq}A $. It follows from Lemma \ref{order}(d) that $B=AP_{\overline{\mathscr{R}(B^*)}} $ and $B=P_{\overline{\mathscr{R}(B)}}A$. Hence, $\mathscr{R}(B)\subseteq\mathscr{R}(A)$ and $\mathscr{R}(B^*)\subseteq\mathscr{R}(A^*)$. It follows from $ \mathscr{R}(B)\subseteq\mathscr{R}(A)$ and Lemma \ref{AX=C} that $ AA^\dagger B=B $. It follows from $ \mathscr{R}(B^*)\subseteq\mathscr{R}(A^*)$ and Lemma \ref{AX=C} that $ BA^\dagger A=\left((A^\dagger A)^*B^*\right)^*=\left(A^*{A^\dagger}^* B^*\right)^*=B $. Hence, $ X=A^\dagger $ is a solution of the system of operator equations $ BXA=B=AXB $.\\
(b) Let $B$ be a closed range operator and $B\stackrel{*}{\leq}A $. It follows from Lemma \ref{order} that $ B=AP_{\overline{\mathscr{R}(B^*)}} $ and $B=P_{\mathscr{R}(B)}A $. Applying (\ref{m1}), we conclude that $ AB^\dagger B=B $ and $BB^\dagger A=B $. Hence, $X=B^\dagger $ is a solution of the system $ BXA=B=AXB $.
\end{proof}

\begin{proposition}\label{remark}
Let $ A, B, X\in\mathbb{B}(\mathscr{H})$.\\
 If $ A\stackrel{*}{\leq}B $ and $BXA=B=AXB $, then $ \mathscr{N}(B)=\mathscr{N}(A), \mathscr{R}(B)=\mathscr{R}(A) {\rm ~and~} AXA=A $.
\end{proposition}
\begin{proof}
 Let $ A\stackrel{*}{\leq}B $ and $BXA=B=AXB $. Applying Lemma \ref{order}(d) we have $A=P_{\overline{\mathscr{R}(A)}}B=BP_{\overline{\mathscr{R}(A^*)}} $. Hence, $ \mathscr{R}(A)\subseteq\mathscr{R}(B)$ and $\mathscr{N}(B)\subseteq\mathscr{N}(A)$. Using Proposition \ref{main},
 \begin{eqnarray*}
 \mathscr{N}(B)=\mathscr{N}(A), \mathscr{R}(B)=\mathscr{R}(A) {\rm ~and~} AXA=A.
 \end{eqnarray*}
\end{proof}
\begin{remark}
Note that the converse of Proposition \ref{remark} is not true, in general. Set $A^\dagger, A^*, A$ instead of $A, B, X$. If $ A\in\mathbb{B}(\mathscr{H})$ has closed range, then, by (\ref{m1}), we have $\mathscr{R}(A^*)=\mathscr{R}(A^\dagger), \mathscr{N}(A^*)=\mathscr{N}(A^\dagger)$ and $A^\dagger AA^\dagger=A^\dagger $ but not $ A^\dagger \stackrel{*}{\leq}A^* $. Indeed, if $ A^\dagger \stackrel{*}{\leq}A^* $, then by utilizing Lemma \ref{order}(d), we have $A^\dagger=P_{\mathscr{R}(A^\dagger)}A^* $. It follows from $ \mathscr{R}(A^\dagger)=\mathscr{R}(A^*)$ that $A^\dagger=P_{\mathscr{R}(A^*)}A^*=A^* $.
\end{remark}
\begin{theorem}
Let $ A, B\in\mathbb{B}(\mathscr{H})$ and $B\stackrel{*}{\leq}A$. Then, the following statements are equivalent:\\
\begin{enumerate}
\item[(a)] There exists a solution $ X\in\mathbb{B}(\mathscr{H})$ of the system $ BXA=B=AXB $;\\
\item[(b)] $B\stackrel{*}{\leq}AXA$\\
\end{enumerate}
\end{theorem}
\begin{proof}
$ (a)\Longrightarrow (b)$: Let $ X\in\mathbb{B}(\mathscr{H})$ is a solution of the system $ BXA=B=AXB $. Hence, $B-BXA=0 $ and $B-AXB=0 $. It follows from the assumption $ B\stackrel{*}{\leq}A $ and Lemma \ref{order}(d) that $ B=P_{\overline{\mathscr{R}(B)}}A $ and $B=AP_{\overline{\mathscr{R}(B^*)}} $. Hence,
 \begin{eqnarray*}
 P_{\overline{\mathscr{R}(B)}}(B-AXA)=B-P_{\overline{\mathscr{R}(B)}}AXA=B-BXA=0
 \end{eqnarray*}
 {\rm and}
 \begin{eqnarray*}
 (B-AXA)P_{\overline{\mathscr{R}(B^*)}}=B-AXAP_{\overline{\mathscr{R}(B^*)}}=B-AXB=0.
 \end{eqnarray*}
 Therefore, $B\stackrel{*}{\leq}AXA $.\\
$ (b)\Longrightarrow (a)$: Suppose that $ B\stackrel{*}{\leq}AXA $. Applying Lemma \ref{order}(d), we infer that $P_{\overline{\mathscr{R}(B)}}(B-AXA)=0 $ and $(B-AXA)P_{\overline{\mathscr{R}(B^*)}}=0 $. It follows from the assumption $ B\stackrel{*}{\leq}A $ and Lemma \ref{order}(d) that $ B=P_{\overline{\mathscr{R}(B)}}A $ and $B=AP_{\overline{\mathscr{R}(B^*)}} $, whence
\begin{eqnarray*}
 B-BXA=B-P_{\overline{\mathscr{R}(B)}}AXA= P_{\overline{\mathscr{R}(B)}}(B-AXA)=0
\end{eqnarray*} and
\begin{eqnarray*}
 B-AXB=B-AXAP_{\overline{\mathscr{R}(B^*)}}= (B-AXA)P_{\overline{\mathscr{R}(B^*)}}=0.
\end{eqnarray*}
Therefore, $X$ is a solution of the system $ BXA=B=AXB $.
\end{proof}

Let $ A, B\in\mathbb{B}(\mathscr{H})$ have closed ranges. It follows from Proposition \ref{star} that $A^\dagger $ and $B^\dagger $ are solutions of the system $ BXA=B=AXB$. Therefore, we are interested in the study of the following system of operator equations:\\
\begin{eqnarray}\label{m31}
BXA=B=AXB;
\end{eqnarray}
\begin{eqnarray}\label{m32}
BAX=B=XAB.
\end{eqnarray}
Let $ A, B\in\mathbb{B}(\mathscr{H})$. An operator $ C\in\mathbb{B}(\mathscr{H})$ is said to be an inverse of $A$ along $B$ if it fulfills one of the equations (\ref{m31}) or (\ref{m32}). If $ A\in\mathbb{B}(\mathscr{H})$ is invertible, then $ X=A^{-1} $ is a solution of the system $ XA=I=AX$. Hence, $A^{-1} $ is an inverse of $A$ along $ I, $ where $I $ is the identity of $ \mathbb{B}(\mathscr{H})$.\\
 Let $ A\in\mathbb{B}(\mathscr{H})$ have closed range. Using (\ref{m1}), we have $AA^\dagger A=A=AA^\dagger A $. Hence, $A^\dagger $ satisfies Eq. (\ref{m31}). Therefore, $A^\dagger $ is the inverse of $A$ along $A$.\\
It follows from (\ref{m1}) that $ A^*AA^\dagger =A^*=A^\dagger AA^* $. Hence, $A^\dagger$ satisfies Eq. (\ref{m32}). Therefore, $A$ is the inverse of $ A $ along $ A^* $.

\begin{lemma}\label{AXB=C}\cite[Theorem 2.1]{Mitra}
Let $ C\in\mathbb{B}(\mathscr{H})$ and $A, B\in\mathbb{B}(\mathscr{H})$ have closed ranges. Then, the equation $ AXB=C $ has a solution $ X\in\mathbb{B}(\mathscr{H})$ if and only if $ \mathscr{R}(C)\subseteq \mathscr{R}(A), \mathscr{R}(C^*)\subseteq \mathscr{R}(B^*)$, and this if and only if $ AA^\dagger CB^\dagger B=C $. In this case, $ X=A^\dagger CB^\dagger+U-A^\dagger AUBB^\dagger $, where $U\in\mathbb{B}(\mathscr{H})$ is arbitrary.
\end{lemma}

In the next result we provide a general solution of the system $ BXA=B=AXB $.
\begin{theorem}
Let $ A, B\in\mathbb{B}(\mathscr{H})$ have closed ranges and $B\stackrel{*}{\leq}A $. Then, the general solution of the system of operator equations $ BXA=B=AXB $ is
\begin{eqnarray*}
 X&=&A^\dagger BB^\dagger +A^\dagger \left[B(I-AA^\dagger)+(A-B)S\right](A-B)^\dagger +T-A^\dagger AT(A-B)^\dagger (A-B)\\
 &&-A^\dagger B(I-AA^\dagger )(A-B)^\dagger BB^\dagger-A^\dagger (A-B)S(A-B)^\dagger BB^\dagger\\
 &&-A^\dagger ATBB^\dagger +A^\dagger AT(A-B)^\dagger (A-B)BB^\dagger.
\end{eqnarray*}
 where $S, T\in\mathbb{B}(\mathscr{H}).$
\end{theorem}

\begin{proof}
Let $ A, B $ have closed ranges. It follows from the assumption $ B\stackrel{*}{\leq}A$ and Lemma \ref{order}(d) that $ B=AP_{\mathscr{R}(B^*)} $. Hence, $ \mathscr{R}(B)\subseteq\mathscr{R}(A)$. Using Lemma \ref{AX=C}, we have $AA^\dagger B=B $. It follows from $ AA^\dagger BB^\dagger B=B $ and Lemma \ref{AXB=C} that the equation $ AXB=B $ is solvable. In this case, the general solution is
\begin{eqnarray}\label{AW}
 X=A^\dagger BB^\dagger+W-A^\dagger AWBB^\dagger,
\end{eqnarray}
where $W\in\mathbb{B}(\mathscr{H})$ is arbitrary. If $X$ satisfies the equation $ BXA=B$, then
\begin{eqnarray*}
B(A^\dagger BB^\dagger+W-A^\dagger AWBB^\dagger)A=B.
\end{eqnarray*}
 It follows from the assumption $ B\stackrel{*}{\leq}A$ and Lemma \ref{order}(d) that $B=P_{\mathscr{R}(B)}A $. Applying (\ref{m1}), $BB^\dagger A=B $. Hence,
 \begin{eqnarray*}
 BA^\dagger B+BWA-BA^\dagger AWB=B.
 \end{eqnarray*}
  Therefore, $B(A^\dagger B+WA-A^\dagger AWB)=B $. So, $A^\dagger B+WA-A^\dagger AWB $ is a solution of the equation $ BX=B $. Utilizing Lemma \ref{AX=C} again, we have
 \begin{eqnarray}\label{ABW}
 A^\dagger B+WA-A^\dagger AWB=B^\dagger B+(I-B^\dagger B)S,
 \end{eqnarray}
where $S\in\mathbb{B}(\mathscr{H})$ is arbitrary. Multiply the left hand side of Eq. (\ref{ABW}) by $ A $, to get
\begin{eqnarray*}
AA^\dagger B+AWA-AA^\dagger AWB=AB^\dagger B+A(I-B^\dagger B)S
\end{eqnarray*}It follows from the assumption $ B\stackrel{*}{\leq}A$ and Lemma \ref{order}(d) that $B=AP_{\mathscr{R}(B^*)} $. Applying (\ref{m1}), $AB^\dagger B =B $. We derive that
\begin{eqnarray*}
AA^\dagger B+AWA-AWB=B+(A-B)S.
\end{eqnarray*}
Now, we get $ AW(A-B)=B(I-AA^\dagger)+(A-B)S $. So, $ W $ is a solution of the equation $ AX(A-B)=B(I-AA^\dagger)+(A-B)S $. Using Lemma \ref{AX=C}, we get that
 \begin{eqnarray*}
 W=A^\dagger \left[B(I-AA^\dagger)+(A-B)S\right](A-B)^\dagger +T-A^\dagger AT(A-B)^\dagger (A-B),
 \end{eqnarray*}
 where $T\in\mathbb{B}(\mathscr{H})$ is arbitrary. By putting $ W $ in Eq. (\ref{AW}), we reach
 \begin{eqnarray*}
 X&=&A^\dagger BB^\dagger +A^\dagger \left[B(I-AA^\dagger)+(A-B)S\right](A-B)^\dagger +T-A^\dagger AT(A-B)^\dagger (A-B)\\
 &&-A^\dagger A( A^\dagger\left[B(I-AA^\dagger)+(A-B)S\right](A-B)^\dagger\\
 &&+T-A^\dagger AT(A-B)^\dagger (A-B) BB^\dagger\\
&=&A^\dagger BB^\dagger +A^\dagger \left[B(I-AA^\dagger)+(A-B)S\right](A-B)^\dagger +T-A^\dagger AT(A-B)^\dagger (A-B)\\
 &&-A^\dagger AA^\dagger B(I-AA^\dagger )(A-B)^\dagger BB^\dagger-A^\dagger AA^\dagger (A-B)S(A-B)^\dagger BB^\dagger\\
 &&-A^\dagger ATBB^\dagger +A^\dagger AT(A-B)^\dagger (A-B)BB^\dagger \qquad \qquad \qquad \qquad ({\rm by~} (\ref{m1}))\\
 &=&A^\dagger BB^\dagger +A^\dagger \left[B(I-AA^\dagger)+(A-B)S\right](A-B)^\dagger +T-A^\dagger AT(A-B)^\dagger (A-B)\\
 &&-A^\dagger B(I-AA^\dagger )(A-B)^\dagger BB^\dagger-A^\dagger (A-B)S(A-B)^\dagger BB^\dagger\\
 &&-A^\dagger ATBB^\dagger +A^\dagger AT(A-B)^\dagger (A-B)BB^\dagger.
 \end{eqnarray*}
 \end{proof}
\begin{theorem}\label{equvalent1}
Let $ A, B\in \mathbb{B}(\mathscr{H})$ where $A$ has closed range. If the system $ BXA=B=AXB $ is solvable, then the system $ XB=A^\dagger B, BX=BA^\dagger $ is solvable. Conversely, If $ B\stackrel{*}{\leq}A $ and the system $ XB=A^\dagger B, BX=BA^\dagger $ is solvable, then the system $ BXA=B=AXB $ is solvable.
\end{theorem}
\begin{proof}
($\Longrightarrow$): Let $ \widetilde{X} $ be a solution of the system $ BXA=B=AXB $. It follows from $ B=A\widetilde{X}B $ that $ \mathscr{R}(B)\subseteq \mathscr{R}(A)$. Using Lemma \ref{AX=C}, $ AA^\dagger B=B $. It follows from (\ref{m1}) that
\begin{eqnarray*}
P_{\overline{\mathscr{R}(A^*)}}\widetilde{X}AA^\dagger B =(A^\dagger A)\widetilde{X}(AA^\dagger)B=(A^\dagger A)\widetilde{X}(AA^\dagger B)=A^\dagger (A\widetilde{X}B)=A^\dagger B.
\end{eqnarray*}So, $ P_{\overline{\mathscr{R}(A^*)}}\widetilde{X}AA^\dagger $ is a solution of the equation $ XB=A^\dagger B $.
 Since $B^*=(B\widetilde{X}A)^*=A^*\widetilde{X}^*B^* $, we have $\mathscr{R}(B^*)\subseteq \mathscr{R}(A^*)$. Applying Lemma \ref{Douglas}, there exists $ Y\in\mathbb{B}(\mathscr{H})$ such that $ B=YA $. Hence,
 \begin{eqnarray*}
 BP_{\overline{\mathscr{R}(A^*)}}\widetilde{X}AA^\dagger&=& B(A^\dagger A)\widetilde{X}(AA^\dagger)=Y(AA^\dagger A)\widetilde{X}(AA^\dagger )\\
 &=& (YA\widetilde{X}A)A^\dagger =(B\widetilde{X}A)A^\dagger =BA^\dagger.
\end{eqnarray*}Therefore, $ P_{\overline{\mathscr{R}(A^*)}}\widetilde{X}AA^\dagger $ is a solution of the equation $ B= BA^\dagger $.
 Thus $ P_{\overline{\mathscr{R}(A^*)}}\widetilde{X}AA^\dagger $ is a solution of the system $ XB=A^\dagger B, BX=BA^\dagger $.\\
($\Longleftarrow$): Suppose that $ \widetilde{X} $ is a solution of the system $ XB=A^\dagger B, BX=BA^\dagger $. It follows from the assumption $ B\stackrel{*}{\leq}A $ that $ B=AP_{\overline{\mathscr{R}(B^*)}} $ and $B=P_{\overline{\mathscr{R}(B)}}A $. Hence, $\mathscr{R}(B)\subseteq\mathscr{R}(A)$ and $\mathscr{R}(B^*)\subseteq\mathscr{R}(A^*)$. It follows from $ \mathscr{R}(B)\subseteq\mathscr{R}(A)$ to Lemma \ref{AX=C} that $ AA^\dagger B=B $. Hence, $A\widetilde{X}B=A(A^\dagger B)=AA^\dagger B=B $. It follows from $ \mathscr{R}(B^*)\subseteq\mathscr{R}(A^*)$ and Lemma \ref{Douglas} that there exists $ Z^*\in\mathbb{B}(\mathscr{H})$ such that $ B=ZA $. Hence,
\begin{eqnarray*}
B\widetilde{X}A=(BA^\dagger )A=BA^\dagger A=ZAA^\dagger A=ZA=B.
\end{eqnarray*}
Therefore, $ \widetilde{X} $ is a solution of the system $BXA=B=AXB$.
\end{proof}
\begin{lemma}\label{koliha}\cite[Theorem 4.2]{Koliha}
Let $ A, B, C,D\in\mathbb{B}(\mathscr{H})$ and $A, B, M=B^*(I-A^\dagger A)$ have closed ranges. Then, the system $ AX=C, \quad XB=D $ have a hermitian solution $ X\in\mathbb{B}(\mathscr{H}) $ if and only if
\begin{eqnarray*}
AA^\dagger C=C, \quad DB^\dagger B=D, \quad AD=CB
\end{eqnarray*}
and $AC^*$ and $B^*D$ are hermitian. In this case, the general hermitian solution is
\begin{eqnarray*}
X &=& A^\dagger C+(I-A^\dagger A)M^\dagger s(T)\\
&&+(I-A^\dagger A)(I-M^\dagger M)\left[A^\dagger C+(I-A^\dagger A)M^\dagger s(T)\right]^*\\
&&+ (I-A^\dagger A)(I-M^\dagger M)W(I-M^\dagger M)^*(I-A^\dagger A)^*,
\end{eqnarray*}
where $W\in\mathbb{B}(\mathscr{H})$ is hermitian and $s(T)=D^*-B^*A^\dagger C $ is the so-called Schur complement of the block matrix $ T=\left [  \begin{array}{cc}
A&C\\
B^*&D^*
\end{array}  \right] $.
\end{lemma}
\begin{theorem}
Suppose that $ A, B\in\mathbb{B}(\mathscr{H}) $ have closed ranges. If $ B\stackrel{*}{\leq}A $ and $ B^*A^\dagger B, B{A^\dagger}^* B^* $ are hermitian, then the system $ BXA=B=AXB $ has a hermitian solution.
\end{theorem}
\begin{proof}
Replace $ A, B, C, D $ in Lemma \ref{koliha} by $ B, B, BA^\dagger, A^\dagger B $ to get
\begin{eqnarray*}
AA^\dagger C=BB^\dagger (BA^\dagger) =BA^\dagger=C, \quad DB^\dagger B=(A^\dagger B)B^\dagger B=A^\dagger B=D \\
\end{eqnarray*} and
\begin{eqnarray*}
AD=B(A^\dagger B)=(BA^\dagger)B=CB, \quad AC^*=B(BA^\dagger)^*=B{A^\dagger}^* B^*, \quad B^*D=B^*A^\dagger B.
\end{eqnarray*}
Using Lemma \ref{koliha}, the system $ XB=A^\dagger B, BX=BA^\dagger $ has a hermitian solution, say, $ \widetilde{X}$. It follows from the assumption $ B\stackrel{*}{\leq}A $ that $ B=AP_{\overline{\mathscr{R}(B^*)}} $ and $ B=P_{\overline{\mathscr{R}(B)}}A $. Hence, $ \mathscr{R}(B)\subseteq\mathscr{R}(A) $ and $ \mathscr{R}(B^*)\subseteq\mathscr{R}(A^*) $. It follows from $ \mathscr{R}(B)\subseteq\mathscr{R}(A) $ and Lemma \ref{AX=C} that $ AA^\dagger B=B $. Hence, $ A\widetilde{X}B=A(A^\dagger B)=AA^\dagger B=B $. It follows from $ \mathscr{R}(B^*)\subseteq\mathscr{R}(A^*) $ and Lemma \ref{Douglas} that there exists $ Z\in\mathbb{B}(\mathscr{H}) $ such that $ B=ZA $. Hence,
\begin{eqnarray*}
B\widetilde{X}A=(BA^\dagger )A=BA^\dagger A=ZAA^\dagger A=ZA=B.
\end{eqnarray*}
Therefore, $ \widetilde{X} $ is a hermitian solution of the system $ BXA=B=AXB $.
\end{proof}

\section{$*$-order via other operator equations}

 Generally speaking, the inequality $ PB\stackrel{*}{\leq}B $ dose not hold for any $ P\in\mathscr{P}(\mathscr{H})$ even if $ \mathscr{R}(P)\subseteq\overline{\mathscr{R}(B)} $. In \cite[Lemma 2.6]{Antezana}, some conditions are mentioned which give a one-sided description of the relation $ A\stackrel{*}{\leq}B $ regarding (\ref{m2}). \\
 The next result is known.
\begin{proposition} \label{p-order} \cite[Proposition 2.6]{Antezana} Let $ B\in\mathbb{B}(\mathscr{H})$.\\
\begin{enumerate}
 \item[(a)] If $ P\in\mathscr{OP}(\mathscr{H})$ and $\mathscr{R}(P)\subseteq\overline{\mathscr{R}(B)}, $ then $ PB\stackrel{*}{\leq}B $ if and only if $ PBB^*=BB^*P$.\\
 \item[(b)] If $ Q\in\mathscr{OP}(\mathscr{H})$ and $\mathscr{R}(Q)\subseteq\overline{\mathscr{R}(B^*)}, $ then $ BQ\stackrel{*}{\leq}B $ if and only if $ QB^*B=B^*BQ$.
\end{enumerate}
 \end{proposition}
 In the following, we state a generalization of Proposition \ref{p-order}.
\begin{proposition}
Let $ B\in\mathbb{B}(\mathscr{H})$. If there exist $ P, Q\in\mathscr{OP}(\mathscr{H})$ such that $ \mathscr{R}(P)\subseteq\overline{\mathscr{R}(B)} $and $ \mathscr{R}(Q)\subseteq\overline{\mathscr{R}(B^*)}, $
then $ PBQ\stackrel{*}{\leq}B $ if and only if $ PBQB^*=BQB^*P $ and $QB^*PB=B^*PBQ$.
\end{proposition}
\begin{proof}
$(\Longrightarrow)$: Let $ PBQ\stackrel{*}{\leq}B $. Applying (\ref{m2}), we get that
\begin{eqnarray*}
 PBQB^*=(PBQ)B^*=B(PBQ)^*=BQB^*P
 \end{eqnarray*} and
 \begin{eqnarray*}
 B^*PBQ=B^*(PBQ)=(PBQ)^*B=QB^*PB.
\end{eqnarray*}
 $( \Longleftarrow)$: Let $ PBQB^*=BQB^*P $ and $QB^*PB=B^*PBQ $. Applying (\ref{m2}), we obtain that
 \begin{eqnarray*}
 (PBQ)(PBQ)^*=PBQB^*P=(BQB^*P)P=BQB^*P=B(PBQ)^*
\end{eqnarray*} and
\begin{eqnarray*}
(PBQ)^*(PBQ)=QB^*PBQ=Q(QB^*PB)=QB^*PB=(PBQ)^*B.
\end{eqnarray*}
\end{proof}

The next known theorem gives a characterization of the order $ \stackrel{*}{\leq} $.
 \begin{theorem}\cite[Theorem 2.3]{Deng}\label{Deng}
 Let $ A\in\mathbb{B}(\mathscr{H})$ and $C\in\mathscr{Q}(\mathscr{H})$. Then, $ C\stackrel{*}{\leq}A $ if and only if there exists $ X\in\mathbb{B}(\mathscr{H})$ such that $ A=C+(I-C^*)X(I-C^*)$.
\end{theorem}

In the following, we establish an analogue of Theorem \ref{Deng} for generalized projections on a Hilbert space. Recall that an operator $ A\in\mathbb{B}(\mathscr{H})$ is a generalized projection if $ A^2=A^*$.

\begin{lemma}\label{generalization} \cite[Theorem A.2]{radosavljevic}
Let $ A\in\mathbb{B}(\mathscr{H}) $ be a generalized projection.Then, $ A $ is a closed range operator and $ A^3 $ is an orthogonal projection on $ \mathscr{R}(A) $. Moreover, $ \mathscr{H} $ has decomposition
\begin{eqnarray*}
\mathscr{H}=\mathscr{R}(A)\bigoplus \mathscr{N}(A)
\end{eqnarray*} and $ A $ has the following matrix representation
\begin{eqnarray*}
  A=\left [  \begin{array}{cc}
 A_{1} & 0\\
  0 & 0
 \end{array}  \right]:
 \left [  \begin{array}{cc}
 \mathscr{R}(A)\\
 \mathscr{N}(A)
 \end{array}  \right]
 \rightarrow
 \left [  \begin{array}{cc}
 \mathscr{R}(A)\\
  \mathscr{N}(A)
 \end{array}  \right],
\end{eqnarray*} where the restriction $ A_{1}=A|_{\mathscr{R}(A)} $ is unitary on $ \mathscr{R}(A) $.
\end{lemma}
\begin{theorem}
 Let $ A\in\mathbb{B}(\mathscr{H})$ and $B\in\mathscr{GP}(\mathscr{H})$. Then, $ B\stackrel{*}{\leq}A $ if and only if there exists $ X\in\mathbb{B}(\mathscr{H})$ such that $ A=B+(I-BB^*)X(I-B^*B)$.
\end{theorem}
\begin{proof}
$ (\Longrightarrow)$: Let $ B\in\mathscr{GP}(\mathscr{H})$ and $B\stackrel{*}{\leq}A $. Employing Lemma \ref{generalization}, we infer that $B$ has closed range and $B^3=P_{\mathscr{R}(B)}$. It follows from (\ref{m1}) that
\begin{eqnarray*}
\mathscr{R}(B^*)=\mathscr{R}(B^*B)=\mathscr{R}(B^3)=\mathscr{R}(BB^*)=\mathscr{R}(B).
\end{eqnarray*} Hence, $P_{\mathscr{R}(B)}=P_{\mathscr{R}(B^*)}=BB^*=B^*B $. Therefore, $P_{\mathscr{N}(B)}=P_{\mathscr{N}(B^*)}=I-BB^*=I-B^*B$.
Applying Lemma \ref{order}(c), we get $A=B+P_{\mathscr{N}(B^*)}AP_{\mathscr{N}(B)} $. Hence, $A=B+(I-BB^*)A(I-B^*B)$.\\
$(\Longleftarrow)$: Let $ X\in\mathbb{B}(\mathscr{H})$ be a solution of the equation $ A=B+(I-BB^*)X(I-B^*B)$. Since $B$ is a generalized projection, so $ B^*BB^*=B^* $. Hence,
 \begin{eqnarray*}
 B^*A=B^*B+B^*(I-BB^*)X(I-B^*B)= B^*B
\end{eqnarray*}and
\begin{eqnarray*}
AB^*=BB^*+(I-BB^*)X(I-B^*B)B^*= BB^*.
\end{eqnarray*} Therefore, $B\stackrel{*}{\leq}A $ by (\ref{m2}).
\end{proof}
In the next result, we show that if $A$ is a generalized projection and $B\stackrel{*}{ \leq}A\stackrel{*}{\wedge}A^* $, then $ AA^*$ can be written as the sum of two idempotents.
\begin{theorem}\label{AA^*}
Let $ A\in\mathscr{GP}(\mathscr{H})$ and $B\in\mathbb{B}(\mathscr{H})$. If $ B\stackrel{*}{ \leq}A\stackrel{*}{\wedge}A^* $, then $B$ is an idempotent and there exist an idempotent $X$ such that $ AA^*=B+X $ and $B^*X=XB^*=0 $.
\end{theorem}
\begin{proof}
 Let $ B\stackrel{*}{ \leq}A\stackrel{*}{\wedge}A^* $. It follows from the assumption $ A^2=A^* $ and Lemma \ref{order}(d) that
\begin{eqnarray*}
 B^2=(P_{\overline{\mathscr{R}(B)}}A^*)(A^*P_{\overline{\mathscr{R}(B^*)}})=P_{\overline{\mathscr{R}(B)}}{A^*}^2P_{\overline{\mathscr{R}(B^*)}}=P_{\overline{\mathscr{R}(B)}}AP_{\overline{\mathscr{R}(B^*)}}=BP_{\overline{\mathscr{R}(B^*)}}=B.
\end{eqnarray*}
Using Lemma \ref{order}, we get that
 \begin{eqnarray*}
 AB=A(AP_{\overline{\mathscr{R}(B^*)}})=A^2P_{\overline{\mathscr{R}(B^*)}}=A^*P_{\overline{\mathscr{R}(B^*)}}=B,
 \end{eqnarray*}
 \begin{eqnarray*}
 BA=(P_{\overline{\mathscr{R}(B)}}A)A=P_{\overline{\mathscr{R}(B)}}A^2=P_{\overline{\mathscr{R}(B)}}A^*=B,
 \end{eqnarray*}
 \begin{eqnarray*}
 A^*B=A^*(A^*P_{\overline{\mathscr{R}(B^*)}})={A^*}^2P_{\overline{\mathscr{R}(B^*)}}=AP_{\overline{\mathscr{R}(B^*)}}=B
\end{eqnarray*} and
 \begin{eqnarray*}
 BA^*=(P_{\overline{\mathscr{R}(B)}}A^*)A^*=P_{\overline{\mathscr{R}(B)}}{A^*}^2=P_{\overline{\mathscr{R}(B)}}A=B.
\end{eqnarray*}
Let $ X=AA^*-B $. It follows from the assumption $ B\stackrel{*}{ \leq}A\stackrel{*}{\wedge}A^* $ that
\begin{eqnarray*}
X^2=(AA^*-B)^2&=&(AA^*)^2+B^2-AA^*B-BAA^*\\
&=&AA^*+B-AB-BA^*\\
&=&AA^*+B-B-B=AA^*-B=X.
\end{eqnarray*}
 Hence, $X$ is an idempotent. Applying (\ref{m2}), we have
\begin{eqnarray*}
 B^*X=B^*(AA^*-B)=B^*AA^*-B^*B=B^*A^*A-B^*B=B^*A-B^*B=0
\end{eqnarray*}and
\begin{eqnarray*}
XB^*=(AA^*-B)B^*=AA^*B^*-BB^*=AB^*-BB^*=0.
\end{eqnarray*}
\end{proof}

\begin{lemma}\label{idempotent}
Let $ A\in\mathscr{Q}(\mathscr{H})$ and $B\in\mathbb{B}(\mathscr{H})$. Then, $ B\stackrel{*}{ \leq}A $ if and only if $B$ is an idempotent and there exists an idempotent $X$ such that $ A=B+X $ and $B^*X=XB^*=0$.
\end{lemma}
\begin{proof}
$( \Longrightarrow)$: Let $ B\stackrel{*}{ \leq}A $. It follows from the assumption $ A^2=A $ and Lemma \ref{order}(d) that
\begin{eqnarray*}
 B^2=(P_{\overline{\mathscr{R}(B)}}A)(AP_{\overline{\mathscr{R}(B^*)}})=P_{\overline{\mathscr{R}(B)}}A^2P_{\overline{\mathscr{R}(B^*)}}=(P_{\overline{\mathscr{R}(B)}}A)P_{\overline{\mathscr{R}(B^*)}}=BP_{\overline{\mathscr{R}(B^*)}}=B.
\end{eqnarray*}
Utilizing Lemma \ref{order}(d), we obtain that
 \begin{eqnarray*}
 AB=A(AP_{\overline{\mathscr{R}(B^*)}})=A^2P_{\overline{\mathscr{R}(B^*)}}=AP_{\overline{\mathscr{R}(B^*)}}=B
 \end{eqnarray*} and
 \begin{eqnarray*}
 BA=(P_{\overline{\mathscr{R}(B)}}A)A=P_{\overline{\mathscr{R}(B)}}A^2=P_{\overline{\mathscr{R}(B)}}A=B.
\end{eqnarray*} Hence, $X=A-B $ is an idempotent and $B^*X=B^*(A-B)=0 $ and $XB^*=(A-B)B^*=0 $.\\
$(\Longleftarrow)$: Let $ A=B+X $ and $B^*X=XB^*=0 $ for some idempotent $X$. Then, $ B^*(A-B)=B^*X=0 $ and $(A-B)B^*=XB^*=0 $. Therefore, $B\stackrel{*}{\leq}A $ by (\ref{m2}).
\end{proof}

\begin{corollary}
Let $ A\in\mathscr{GP}(\mathscr{H})$ and $B\in\mathbb{B}(\mathscr{H})$. Then, $ B\stackrel{*}{ \leq}AA^* $ if and only if $B$ is an idempotent and there exists an idempotent $X$ such that $ AA^*=B+X $ and $B^*X=XB^*=0$.
\end{corollary}\label{corollary}
\begin{proof}
Let $ A\in\mathscr{GP}(\mathscr{H})$. Then, $ (AA^*)^2=AA^*AA^*=AA^* $. Hence, $AA^* $ is an idempotent. Now apply Lemma \ref{idempotent}.
\end{proof} 

We end our work with the following result.

\begin{proposition}
Let $ A\in\mathbb{B}(\mathscr{H})$ and $C\in\mathscr{GP}(\mathscr{H})$. Then, $ B\in\mathbb{B}(\mathscr{H})$ is common $ *- $ lower bound of $A$ and $CC^* $ if and only if $B$ is an idempotent and there exist $ X, Y\in\mathbb{B}(\mathscr{H})$ such that
\begin{eqnarray*}
A=B+(I-B^*)X(I-B^*), {\rm ~and~} CC^*=B+Y,
\end{eqnarray*}
where $B^*Y=YB^*=0 $.
\end{proposition}
\begin{proof}
$(\Longrightarrow)$: If $B$ be a common $ *- $ lower bound of $A$ and $CC^* $, then $ B\stackrel{*}{\leq}A $ and $B\stackrel{*}{\leq}CC^* $. It follows from the assumption $ B\stackrel{*}{\leq}CC^* $ and Lemma \ref{idempotent} that $B$ is an idempotent and there exists an idempotent $ Y\in\mathbb{B}(\mathscr{H})$ such that $ CC^*=B+R $, where $B^*R=RB^*=0 $. Since $B$ is an idempotent and $B\stackrel{*}{\leq}A $, by Theorem \ref{Deng}, there exists $ S\in\mathbb{B}(\mathscr{H})$ such that $ A=B+(I-B^*)S(I-B^*)$.\\
$ (\Longleftarrow)$: If there exists an idempotent $Y$ such that $ CC^*=B+Y $ with $ B^*Y=0 $ and $YB^*=0 $, then $ B\stackrel{*}{\leq}CC^* $. The assumption $ A=B+(I-B^*)S(I-B^*)$ and the fact that $B$ is an idempotent yield $B^*(A-B)=0 $ and $(A-B)B^*=0 $. Hence, $B\stackrel{*}{\leq}A $ and $B$ is a common $ *- $ lower bound of $A$ and $CC^* $.
\end{proof}


\end{document}